\input ./style/arxiv-vmsta.cfg
\documentclass[numbers,compress,v1.0.1]{vmsta}

\usepackage{cases,booktabs,enumerate,array,dsfont}

\volume{3}% Updated by VTEXPTS2LaTeX.exe, 03.01.2017 11:09
\issue{4}% Updated by VTEXPTS2LaTeX.exe, 03.01.2017 11:09
\pubyear{2016}
\firstpage{269}% Updated by VTEXPTS2LaTeX.exe, 03.01.2017 11:09
\lastpage{285}% Updated by VTEXPTS2LaTeX.exe, 03.01.2017 11:09
\doi{10.15559/16-VMSTA66}% Updated by VTEXPTS2LaTeX.exe, 07.12.2016
\startlocaldefs

\allowdisplaybreaks

\newtheorem{prop}{Proposition}
\newtheorem{lemma}{Lemma}

\newtheorem{thm}{Theorem}
\newtheorem{example}{Example}
\theoremstyle{remark}
\newtheorem{remark}{Remark}

\newcommand{\F}{\mathcal F}
\newcommand{\E}{\mathbb E}
\newcommand{\R}{\mathbb R}
\renewcommand{\P}{\mathbb P}
\newcommand{\ind}{\mathds1}
\newcommand*{\abs}[1]{\left\lvert#1\right\rvert}
\newcommand*{\set}[1]{\left\{#1\right\}}

\newcommand{\lloc}{L^1_{\mathrm{loc}}}

\newcolumntype{Q}{>{$}l<{$}}
\endlocaldefs

\begin{document}
\begin{frontmatter}

\title{Drift parameter estimation in stochastic differential equation
with multiplicative stochastic volatility}
%\author[]{\inits{}\fnm{}\snm{}\corref{cor1}}\email{}
%\cortext[cor1]{Corresponding author.}
%
%\author[]{\inits{}\fnm{}\snm{}}\email{}
%
%%\fnref{f1}
%%\fntext[]{Some remarks}
%
%\address[]{}
%\address[]{}

\author[a]{\inits{M.}\fnm{Meriem}\snm{Bel Hadj Khlifa}}\email
{meriem.bhk@outlook.fr}
\author[b]{\inits{Yu.}\fnm{Yuliya}\snm{Mishura}}\email{myus@univ.kiev.ua}
\author[b]{\inits{K.}\fnm{Kostiantyn}\snm{Ralchenko}\corref{cor1}}\email
{k.ralchenko@gmail.com}
\cortext[cor1]{Corresponding author.}
\author[a]{\inits{M.}\fnm{Mounir}\snm{Zili}}\email{Mounir.Zili@fsm.rnu.tn}

\address[a]{University of Monastir, Faculty of Sciences of Monastir,
Department of Mathematics, Avenue de l'Environnement, 5000, Monastir, Tunisia}
\address[b]{Department of Probability Theory,
Statistics and Actuarial Mathematics,\\
Taras Shevchenko National University of Kyiv,\\
64 Volodymyrska, 01601 Kyiv, Ukraine}

\markboth{M. Bel Hadj Khlifa et al.}{Drift parameter
estimation in stochastic differential equation with multiplicative
stochastic volatility}

\begin{abstract}
We consider a stochastic differential equation of the form
\[
dX_t = \theta a(t,X_t)\,dt + \sigma_1(t,X_t)
\sigma_2(t,Y_t)\,dW_t
\]
with multiplicative stochastic volatility, where $Y$ is some adapted
stochastic process.
We prove existence--uniqueness results for weak and strong solutions of
this equation under various conditions on the process $Y$ and the
coefficients $a$, $\sigma_1$, and $\sigma_2$.
Also, we study the strong consistency of the maximum likelihood
estimator for the unknown parameter $\theta$. We suppose that $Y$ is in
turn a solution of some diffusion SDE. Several examples of the main
equation and of the process $Y$ are provided supplying the strong consistency.
\end{abstract}


%\begin{keyword} . \sep.
%\MSC[2010] . \sep.
%\end{keyword}

%
\received{9 November 2016}% Updated by VTEXPTS2LaTeX.exe, 07.12.2016
%13:58
%
\revised{3 December 2016}% Updated by VTEXPTS2LaTeX.exe, 07.12.2016
%13:58
%
\accepted{4 December 2016}% Updated by VTEXPTS2LaTeX.exe, 07.12.2016
%13:58
\publishedonline{13 December 2016}
\end{frontmatter}

\section{Introduction}
The goal of the paper is to study the
stochastic differential equation (SDE), the diffusion coefficient of
which includes an additional stochastic process:
\begin{equation}
\label{eq:intro} dX_t = \theta a(t,X_t)\,dt +
\sigma(t,X_t, Y_t)\,dW_t,
\end{equation}
where $\sigma(t,x,y) = \sigma_1(t,x) \sigma_2(t,y)$, and to
estimate the drift parameter $\theta$ by the observations of stochastic
processes $X$ and $Y$.
Such equations often arise as models of a financial market in
mathematical finance.
For example, one of the first models of such a type with $\sigma(t,x,
y)= xy$ was proposed in \cite{hull}, where $Y$ was the square root of
the geometric Brownian motion process.
A similar model was considered by Heston \cite{hes}, where the
volatility was governed by the Ornstein--Uhlenbeck process.
Fouque et al.\ used the model with stochastic volatility driven by the
Cox--Ingersoll--Ross process; see \cite{fou1,fou2}.
The case where $\sigma(t,x, y)= x \sigma_2(y)$ and $Y$ is the
Ornstein--Uhlenbeck process was studied in \cite{kuc1,kuc2}.

In the present paper, we investigate the existence and uniqueness of
weak and strong solutions to the equation~\eqref{eq:intro}.
We adapt the approaches of Skorokhod~\cite{Skorokhod65}, Stroock and
Varadhan~\cite{stroock,stroock2}, and Krylov~\cite{kryy,KRY} to
establish the weak existence and weak uniqueness. Concerning the strong
existence and uniqueness, we use the well-known approaches of Yamada
and Watanabe~\cite{wat1} (see also \cite{Alta}) for inhomogeneous
coefficients and Lipschitz conditions. In the present paper, we
consider only the case of multiplicative stochastic volatility, where,
as it was mentioned, the diffusion coefficient is factorized as $\sigma
(t,x,y) = \sigma_1(t,x) \sigma_2(t,y)$.
Then we construct the maximum likelihood estimator for the unknown
drift parameter and prove its strong consistency.
As an example, we consider a linear model with stochastic volatility
driven by a solution to some It\^{o}'s SDE. In particular, we study in
details an SDE with constant coefficients, the Ornstein--Uhlenbeck
process, and the geometric Brownian motion, as the model for volatility
(note that process $Y$ can be interpreted not only as a volatility, but
also as an additional source of randomness).
Note that the maximum likelihood estimation in the Ornstein--Uhlenbeck
model with stochastic volatility was studied in \cite{ait-sahalia}.
Similar statistical methods for the case of deterministic volatility
can be found in \cite{hey,jan,liptser-stat2,mis}.

The paper is organized as follows. In Section~\ref{sec:2}, we prove
the existence of weak and strong solutions under different conditions. In
Section~\ref{sec:3}, we establish the strong consistency of the
maximum likelihood estimator of the unknown drift parameter
$\theta$. Section~\ref{sec:4} contains the illustrations of our results
with some
simulations. Auxiliary statements are gathered in Section~\ref{app}.

\section{Existence and uniqueness results for weak and strong solutions}
\label{sec:2}

Let $(\varOmega, \F,\overline{\F}, \P)$ be a complete probability
space with filtration $\overline{\F}=\set{\F_t, t\ge0}$
satisfying the standard assumptions. We assume that all
processes under consideration are adapted to the filtration
$\overline{\F}$.

\subsection{Existence of weak solution in terms of Skorokhod conditions}

Consider the following stochastic differential equation:
\begin{equation}
\label{eq.1.} dX_t = a(t,X_t) \,dt +
\sigma_1(t,X_t) \sigma_2(t,Y_t )\,
dW_t,
\end{equation}
where $X|_{t=0}=X_0\in\R$, $W$ is a Wiener process, and $Y$ is some
adapted stochastic process to be specified later.
\begin{thm}
Let $Y$ be a measurable and continuous process, $a$, $\sigma_1$,
and $\sigma_2$ be continuous w.r.t.\ $x\in\R$, $y\in
\R$, and $t\in[0,T]$, $\sigma_2$ be bounded, and
\[
\big|\sigma_1(t,x)\big|^2 +\big|a(t,x)\big|^2 \leq K
\bigl(1+\abs{x} ^2 \bigr)
\]
for some constant $K>0$.
Then Eq.~\eqref{eq.1.} has a weak solution.
\end{thm}
\begin{proof} Consider a sequence of partitions of $[0,T]$: $0=t_0^n <t_1^n
<\cdots<t_n^n=T$ such that $\lim_{n\rightarrow\infty} \max_k
(t_{k+1}^n-t_k^n )=0$. Define $\xi_k^n$ by $\xi_0^n=X(0)$ and
\[
\xi_{k+1}^n=\xi_k^n+a
\bigl(t_k^n,\xi_k^n\bigr)\Delta
t_k^n+\sigma_1\bigl(t_k^n,
\xi_k^n\bigr) \sigma_2\bigl(t_k^n,Y
\bigl(t_k^n\bigr)\bigr)\Delta W_k^n.
\]
It follows from Lemma \ref{lem1}, Lemma \ref{lem 2}, and Proposition
\ref{prop} in Section~\ref{app} that it is possible to choose a
subsequence $n'$ and construct processes $\widetilde\xi_{n'}$,
$\widetilde W_{n'}$, and $\widetilde Y_{n'}$ such that the
finite-dimensional distributions of $\widetilde\xi_{n'}$, $\widetilde
W_{n'}$, and $\widetilde Y_{n'}$ coincide with those of $\xi^{n'}$,
$W$, and $Y$ and $\widetilde\xi_{n'}\rightarrow\widetilde\xi$,
$\widetilde W_{n'}\rightarrow\widetilde{W}$ and $\widetilde
Y_{n'}\rightarrow\widetilde{Y}$ in probability, where $\widetilde{\xi
}$, $\widetilde{W}$, and $\widetilde{Y}$ are some stochastic processes
(evidently, $\widetilde{W}$ is a Wiener process).
It suffices to prove that $\widetilde{\xi}$ is a solution of Eq.~(\ref
{eq.1.}) when $W$ and $Y$ are replaced by $\widetilde{W}$ and
$\widetilde{Y}$.

We have that $\widetilde{\xi}_{n'}$ satisfies the equation
\begin{align*}
\widetilde\xi_{n'}(t) &= \widetilde\xi_{n'}(0) + \sum
_{t_{k+1}^{n'}\leq
t} a \bigl(t_k^{n'},
\widetilde\xi_{n'} \bigl(t_k^{n'} \bigr) \bigr)
\Delta t_k^{n'}
\\
&\quad+\sum_{t_{k+1}^{n'}\leq t}\sigma_1
\bigl(t_k^{n'},\widetilde\xi_{n'}
\bigl(t_k^{n'} \bigr) \bigr) \sigma_2
\bigl(t_k^{n'},\widetilde{Y}_{n'}
\bigl(t_k^{n'} \bigr) \bigr) \Delta\widetilde{W}_k^{n'}.
\end{align*}
Since $\sigma_2$ is bounded and $\sigma_1$ is of linear growth, their
product is of linear growth:
\[
\big|\sigma_1(t,x)\sigma_2(t,y)\big| \leq K_1(1+
\abs{ x }),
\]
where $K_1>0 $ is a constant. Therefore,
\begin{align*}
&\P \Bigl(\sup_{0 \leq t\leq T}\big|\sigma_1 \bigl(t,
\widetilde{\xi }_{n'}(t) \bigr) \sigma_2 \bigl(t,
\widetilde{Y}_{n'}(t) \bigr)\big|> C \Bigr)
\\
&\quad\leq\P \Bigl(\sup_{0 \leq t\leq T} K_1 \bigl(1+\big|
\widetilde\xi _{n'}(t) \big| \bigr)>C \Bigr) =\P \biggl(\sup
_{0 \leq t\leq T} \big| \widetilde\xi_{n'}(t) \big|>
\frac
{C}{K_1}-1 \biggr)
\\
&\quad=\P \biggl(\sup_{0 \leq t\leq T} \big| \xi^{n'}(t) \big|>
\frac
{C}{K_1}-1 \biggr).
\end{align*}
Using Lemma \ref{lem1}, we get that
\[
\P \Bigl(\sup_{0 \leq
t\leq T}\big| \sigma_1 \bigl(t,
\widetilde\xi_{n'}(t) \bigr) \sigma_2 \bigl(t,\widetilde
Y_{n'}(t) \bigr)\big| >C \Bigr)\rightarrow0 \quad\text{as }C\rightarrow
\infty.\vadjust{\eject}
\]
Moreover, we have that
$\sigma_1(t,x)\sigma_2(t,y)$ is continuous w.r.t.\ $t \in[0,T]$, $x,y
\in\R$. Then, for any $\varepsilon>0$, there exists $\delta>0$ such that
\[
\big|\sigma_1(t_1,x_1)\sigma_2(t_1,y_1)-
\sigma_1(t_2,x_2)\sigma _2(t_2,y_2)\big|
< \varepsilon
\]
whenever
$\abs{t_1-t_2} < \delta$, $\abs{x_1-x_2} < \delta$, $\abs{y_1-y_2} <
\delta$.
Therefore,
\begin{align*}
\P& \bigl(\big| \sigma_1 \bigl(t_1,\widetilde
\xi_{n'}(t_1) \bigr) \sigma_2
\bigl(t_1,\widetilde{Y}_{n'}(t_1) \bigr)-
\sigma_1 \bigl(t_2,\widetilde\xi_{n'}(t_2)
\bigr) \sigma_2 \bigl(t_2,\widetilde{Y}_{n'}(t_2)
\bigr)\big| > \varepsilon \bigr)
\\
&\leq\P \bigl(\big| \widetilde\xi_{n'}(t_1)-\widetilde
\xi_{n'}(t_2)\big| < \delta, \big|\widetilde{Y}_{n'}(t_1)-
\widetilde{Y}_{n'}(t_2)\big| < \delta ,
\\
&\quad \big|\sigma_1 \bigl(t_1,\widetilde{
\xi}_{n'}(t_1) \bigr) \sigma_2
\bigl(t_1,\widetilde{Y}_{n'}(t_1) \bigr) -
\sigma_1 \bigl(t_2,\widetilde\xi_{n'}(t_2)
\bigr) \sigma_2 \bigl(t_2,\widetilde{Y}_{n'}(t_2)
\bigr)\big| > \varepsilon \bigr)
\\
&\quad+ \P \bigl(\big|\widetilde\xi_{n'}(t_1)-\widetilde
\xi_{n'}(t_2)\big| \geq\delta \bigr) +\P \bigl(\big|
\widetilde{Y}_{n'}(t_1)-\widetilde{Y}_{n'}(t_2)
\big| \geq \delta \bigr)
\\
&=\P \bigl(\abs{\xi_{n'}(t_1)-\xi_{n'}(t_2)}
\geq\delta \bigr) + \P \bigl(\abs{ Y(t_1)-Y(t_2)}\geq
\delta \bigr),
\end{align*}
and the last relation implies the following one:
\begin{align*}
\lim_{h\rightarrow0 }\lim_{n'\rightarrow\infty}\sup
_{\abs{ t_1-t_2}
\leq h} \P &\bigl(\big\lvert\sigma_1 \bigl(t_1,
\widetilde\xi_{n'}(t_1) \bigr) \sigma_2
\bigl(t_1,\widetilde{Y}_{n'}(t_1) \bigr)
\\
&- \sigma_1 \bigl(t_2,\widetilde\xi_{n'}(t_2)
\bigr) \sigma_2 \bigl(t_2,\widetilde{Y}_{n'}(t_2)
\bigr)\big\rvert > \varepsilon \bigr)=0.
\end{align*}
Applying Lemma \ref{lem1}, we get that
\begin{align*}
\sum_{t_{k+1}^{n'}\leq t}\sigma_1 \bigl(t_k^{n'},
\widetilde\xi_{n'} \bigl(t_k^{n'} \bigr) \bigr)
\sigma_2 \bigl(t_k^{n'},\widetilde{Y}_{n'}
\bigl(t_k^{n'} \bigr) \bigr) \Delta W_k^{n'}
\rightarrow\int_0^T \sigma_1
\bigl(s,\widetilde{\xi}(s) \bigr) \sigma_2 \bigl(s,\widetilde{Y}(s)
\bigr)\,d\widetilde{W}(s)
\end{align*}
in probability as
$n'\rightarrow\infty$, and we also have that
\[
\sum_{t_{k+1}^{n'}\leq t}a \bigl(t_k^{n'},
\widetilde\xi_{n'} \bigl(t_k^{n'} \bigr) \bigr)
\Delta t_k^{n'} \rightarrow\int_0^T
a \bigl(s,\widetilde{\xi}(s) \bigr)\,ds,
\]
whence the proof follows.
\end{proof}

\subsection{Existence and uniqueness of weak solution in terms of
Stroock--Varadhan conditions}

In this approach, we assume additionally that the process $Y$ also is a
solution of some diffusion stochastic differential equation. Let $W^1$
and $W^2$ be two Wiener processes, possibly correlated, so that
$dW^1_tdW^2_t=\rho dt $ for some $|\rho|\leq1$. In this case, we can
represent $W^2_t=\rho W^1_t+\sqrt{1-\rho^2}W^3_t$, where $W^3$ is a
Wiener process independent of $W^1$.

\begin{thm}
Consider the system of stochastic differential equations
\begin{numcases}{}
dX_t = a(t,X_t)\,dt + \sigma_1(t,X_t)\sigma_2(t, Y_t)\,dW_t^1,\label
{eq:1'}\\
\,dY_t = \alpha(t, Y_t)\,dt + \beta(t, Y_t)\,dW_t^2,\label{eq:1}
\end{numcases}
where all coefficients $a$, $\alpha$, $\sigma_1$, $\sigma_2$, and $\beta
$ are nonrandom measurable and
bounded functions, $\sigma_1$, $\sigma_2$, and $\beta$ are continuous
in all arguments.
Let $|\rho|<1$, and let $\sigma_1(t,x)>0$, $\sigma_2(t,y)>0$, $\beta
(t,y)>0$ for all $t,x,y$.
Then the weak existence and uniqueness in law hold for system \eqref
{eq:1'}--\eqref{eq:1}, and in particular, the weak existence and
uniqueness in law hold for Eq.~\eqref{eq:1'} with $Y$ being a weak
solution of Eq.~\eqref{eq:1}.\vadjust{\eject}
\end{thm}
\begin{proof}
Equations (\ref{eq:1'}) and (\ref{eq:1}) are equivalent to the
two-dimensional stochastic differential equation
\[
dZ(t) = A(t,Z_t)\,dt + B(t,Z_t)\,dW(t),
\]
where $Z(t) =  (
\begin{smallmatrix}
X(t) \\
Y(t)
\end{smallmatrix}
 ) $, $\displaystyle W(t) =  (
\begin{smallmatrix}
W^1(t) \\
W^3(t)
\end{smallmatrix}
 ) $ is a two dimensional Wiener process,
\[
A(t,x,y) \,{=} %
\begin{pmatrix}
a(t,x) \\
\alpha(t,y)
\end{pmatrix} %
,\quad \mbox{and}\quad B(t,x,y)\,{=}
\begin{pmatrix}
\sigma_1(t,x) \sigma_2(t,y) & 0 \\
\rho\beta(t, y) & \sqrt{1-\rho^2}\beta(t,y)
\end{pmatrix} %
.
\]
It follows from the measurability and boundedness of $a$ and $\alpha$
and from the continuity and boundedness of $\sigma_1$, $\sigma_2$, and
$\beta$ that the coefficients of matrices $A$ and $B$ are
nonrandom, measurable, and bounded, and additionally the coefficients
of $B$ are continuous in all arguments.
Then we can apply Theorems 4.2 and 5.6 from \cite{stroock} (see also
Prop.~$1.14$ in \cite{chrn}) and deduce that we have to prove the
following relation: for any $(t,x,y)\in\R^+\times\R^2$, there exists
$\varepsilon(t,x,y)>0$ such that, for all $\lambda\in\R^2$,
\begin{equation}
\label{ineq1}\big\|B(t,x,y)\lambda\big\| \geq\varepsilon (t,x,y)\|\lambda\|.
\end{equation}
Relation~\eqref{ineq1} is equivalent to the following one (we omit arguments):
\[
\sigma_1^2\sigma_2^2
\lambda_1^2 +\beta^2 \bigl(\rho
\lambda_1+\sqrt{1-\rho ^2}\lambda_2
\bigr)^2 \ge\varepsilon^2 \bigl(\lambda_1^2+
\lambda_2^2 \bigr)
\]
or
\begin{equation}
\label{7} \bigl(\sigma_1^2\sigma_2^2+
\beta^2\rho^2 \bigr)\lambda_1^2 +
\beta^2 \bigl(1-\rho^2 \bigr)\lambda_2^2
+2\rho\sqrt{1-\rho^2}\beta^2\lambda_1
\lambda_2 \ge\varepsilon^2 \bigl(\lambda_1^2+
\lambda_2^2 \bigr).
\end{equation}
The quadratic form
\[
Q(\lambda_1,\lambda_2)= \bigl(\sigma_1^2
\sigma_2^2+\beta^2\rho^2 \bigr)
\lambda_1^2 +\beta^2 \bigl(1-
\rho^2 \bigr)\lambda_2^2 +2\rho\sqrt{1-
\rho^2}\beta^2\lambda_1\lambda_2
\]
in the left-hand side of \eqref{7} is positive definite since its discriminant
\[
D=\rho^2 \bigl(1-\rho^2 \bigr)\beta^4-
\beta^2 \bigl(1-\rho^2 \bigr) \bigl(\sigma_1^2
\sigma_2^2+\beta^2\rho^2 \bigr) =-
\beta^2 \bigl(1-\rho^2 \bigr) \sigma_1^2
\sigma_2^2<0.
\]
The continuity of $Q(\lambda_1,\lambda_2)$ implies the existence of
$\min_{\lambda_1^2+\lambda_2^2=1}Q(\lambda_1,\lambda_2)>0$.
Then, putting
$\varepsilon=\min_{\lambda_1^2+\lambda_2^2=1}Q(\lambda_1,\lambda_2)$
and using homogeneity, we get~\eqref{7}.
\end{proof}

\subsection{Existence of strong solution in terms of Yamada--Watanabe
conditions}
Now we consider strong existence--uniqueness conditions for Eq.~\eqref
{eq.1.}, adapting the Yamada--Watanabe conditions for inhomogeneous
coefficients from \cite{Alta}.
\begin{thm}
Let $a$, $\sigma_1$, and $\sigma_2$ be nonrandom measurable
bounded functions such that
\begin{enumerate}[\rm(i)]
\item There exists a positive increasing function $\rho(u)$, $u\in
(0,\infty)$, satisfying $\rho(0)=0$ such that
\[
\big| \sigma_1(t,x)-\sigma_1(t,y)\big|\leq\rho(\abs{x-y }),
\quad t \ge0,\ x, y \in\R,
\]
and
$\int_0^\infty\rho^{-2}(u)du=+\infty$.
\item There exists a positive increasing concave function $k(u)$,
$u\in(0,\infty)$, satisfying $k(0)=0$ such that\vadjust{\eject}
\[
\big|a(t,x)-a(t,y)\big|\leq k(\abs{ x-y }),\quad t \ge0,\ x, y \in{\mathbb R},
\]
and $\int_0^\infty k^{-1}(u)du=+\infty$. Also, let $Y$ be an adapted
continuous stochastic process.
Then the pathwise uniqueness of solution holds for Eq.~\eqref{eq.1.},
and hence it has a unique strong solution.
\end{enumerate}
\end{thm}
\begin{proof}
Let $1>a_1>a_2>\cdots>a_n>\cdots>0$ be defined by
\[
\int_{a_1}^1\rho^{-2}(u)\,du=1,\; \int
_{a_2}^{a_1}\rho^{-2}(u)\, du=2,\ldots,\int
_{a_n}^{a_{n-1}}\rho^{-2}(u)\,du=n,\ldots.
\]
We have that $a_n\rightarrow0$ as $n\rightarrow\infty$. Let $\phi
_n(u)$, $n=1,2,\ldots$, be a continuous function with support contained
in $(a_n,a_{n-1})$ such that $0\leq\phi_n(u) \leq\frac{ 2 \rho
^{-2}(u)}{n}$ and $\int_{a_n}^{a_{n-1}}\phi_n(u)\,du=1$. Such a
function obviously exists. Set
\[
\varphi_n(x)=\int_0^{\abs{x}}\int
_0^y \phi_n(u)\,du\,dy, \quad x\in {
\mathbb R}.
\]
Clearly, $\varphi_n\in C^2({\mathbb R})$, $\abs{\varphi_n'(x)} \leq1$,
and $\varphi_n(x)\nearrow\abs{x}$ as $n\rightarrow\infty$.

Let $X_1 $ and $X_2$ be two solutions of Eq.~(\ref{eq.1.}) on
the same probability space with the same Wiener process and such that
$X_1(0)=X_2(0) $. Then we can present their difference as
\begin{align*}
X_1(t)-X_2(t)&=\int_0^t
\sigma_2\bigl(s,Y(s)\bigr) \bigl(\sigma_1
\bigl(s,X_1(s)\bigr)-\sigma _1\bigl(s,X_2(s)
\bigr)\bigr)\,dW(s)
\\
&\quad+\int_0^t \bigl(a\bigl(s,X_1(s)
\bigr)-a\bigl(s,X_2(s)\bigr)\bigr)\,ds.
\end{align*}
By the It\^o formula,
\begin{align*}
&\varphi_n\bigl(X_1(t)-X_2(t)\bigr)
\\
&\quad=\int_0^t \varphi_n'
\bigl(X_1(s)-X_2(s)\bigr) \sigma_2
\bigl(s,Y(s)\bigr) \bigl(\sigma _1\bigl(s,X_1(s)\bigr)-
\sigma_1\bigl(s,X_2(s)\bigr)\bigr)\,dW(s)
\\
&\qquad+\int_0^t \varphi_n'
\bigl(X_1(s)-X_2(s)\bigr) \bigl(a\bigl(s,X_1(s)
\bigr)-a\bigl(s,X_2(s)\bigr)\bigr)\,ds
\\
&\qquad+\frac{1}{2} \int_0^t
\varphi_n''\bigl(X_1(s)-X_2(s)
\bigr) \sigma_2\bigl(s,Y(s)\bigr)^2\bigl(
\sigma_1\bigl(s,X_1(s)\bigr)-\sigma_1
\bigl(s,X_2(s)\bigr)\bigr)^2\,ds
\\
&\quad=J_1+J_2+J_3.
\end{align*}
We have that $\E(J_1)=0$,
\begin{align*}
\abs{\E(J_2)}&\leq\int_0^t\E
\abs{a\bigl(s,X_1(s)\bigr)-a\bigl(s,X_2(s)\bigr)}\,ds
\\
&\leq\int_0^t \E(k\bigl(\abs{X_1(s)-X_2(s)}
\bigr)\,ds \leq\int_0^t k(\E\bigl(
\abs{X_1(s)-X_2(s)}\bigr)\,ds
\end{align*}
by Jensen's inequality, and
\begin{align*}
\abs{\E(J_3)}&\leq\frac{C^2}{2}\int_0^t
\E \biggl(\frac{2}{n}\rho ^{-2}\bigl(\big|X_1(s)-X_2(s)\big|
\bigr)\rho^2\bigl(\big|X_1(s)-X_2(s)\big|\bigr)
\biggr)ds
\\
&\leq\frac{t}{n}\rightarrow0\quad\text{as }n\rightarrow\infty.
\end{align*}
So by letting $n\rightarrow\infty$ we get
\[
\E\bigl(\big|X_1(s)-X_2(s)\big|\bigr) \leq\int
_0^t k(\E\bigl(\big| X_1(s)-X_2(s)\big|
\bigr)\,ds.
\]
We have that $\int_0^\infty k^{-1}(u)du=+\infty$. Then we get $\E(\abs{
X_1(s)-X_2(s) })=0$, and hence $X_1(s)=X_2(s)$ a.s.
\end{proof}

\subsection{Existence and uniqueness for strong solution in terms of
Lipschitz conditions}

\begin{thm}\label{th:Lip} Let $a$, $\sigma_1 $, and
$\sigma_2$ be nonrandom measurable functions, and let $Y$ be an
adapted continuous stochastic process.
Consider the following assumptions:
\begin{enumerate}[\rm(i)]
\item
There exists $K>0$ such that, for all $t\geq0$ and $x\in\R$,
\[
\big| \sigma_1(t,x)\big| ^2 +\big| a(t,x)\big| ^2
\leq K^2 \bigl(1+\abs{ x }^2\bigr);
\]
\item
For any $n\in{\mathbb N}$, there exists $K_N>0$ such that, for all
$t\geq0$ and for all $(x,y)$ satisfying $\abs{x} \leq N$ and $\abs{y}
\leq N$,
\[
\big| a(t,x)-a(t,y)\big| + \big| \sigma_1(t,x)-\sigma_1(t,y)\big|
\leq K_N \abs { x-y };
\]
\item
$\sup_{s \geq0} \sup_{\abs{x}\leq N} \abs{ \sigma_2(s,x) }\leq C_N$.
\end{enumerate}
Then Eq.~\eqref{eq.1.} has a unique strong solution.
\end{thm}
This result can be proved by using the successive approximation method;
see, for example, \cite[Thm.~1.2]{Nisio15}.

\section{Drift parameter estimation}
\label{sec:3}

\subsection{General results}
Let $(\varOmega, \F,\overline{\F}, \P)$ be a complete probability
space with filtration $\overline{\F}=\set{\F_t, t \ge0}$
satisfying the standard assumptions. We assume that all
processes under consideration are adapted to the filtration
$\overline{\F}$.
Consider a parameterized version of Eq.~\eqref{eq.1.}
\begin{equation}
\label{eq:sde} dX_t =\theta a(t,X_t)\,dt +
\sigma_1(t,X_t) \sigma_2(t,Y_t)
\,dW_t,
\end{equation}
where $W$ is a Wiener process.
Assume that Eq.~\eqref{eq:sde} has a unique strong
solution $X=\{X_t,t\in[0,T]\}$. Our main problem is to estimate
the unknown parameter $\theta$ by continuous observations of $X$ and $Y$.

Denote
\[
f(t,x,y)=\frac{a(t,x)}{\sigma_1^2(t,x)\sigma_2^2(t,y)},\qquad g(t,x,y)=\frac{a(t,x)}{\sigma_1(t,x)\sigma_2(t,y)}.
\]
Assume that, for all $t>0$,
\begin{gather}
\sigma_1(t,X_t)\sigma_2(t,Y_t)
\ne0 \quad\text{a.s.},\label{eq:ne0}
\\
\int_0^t g^2(s,X_s,Y_s)
\,ds<\infty \quad\text{a.s.},\label{eq:sq-int}
\\
\int_0^\infty g^2(s,X_s,Y_s)
\,ds=\infty \quad\text{a.s.}\label{eq:cond}
\end{gather}
Then a likelihood function for Eq.~\eqref{eq:sde} has the form
\[
\frac{dP_{\theta}(T)}{dP_{0}(T)}= \exp\set{\theta\int_0^Tf(t,X_t,Y_t)
\,dX_t -\frac{\theta^{2}}{2}\int_0^Tg^2(t,X_t,Y_t)
\,dt};
\]
see~\cite[Ch.~7]{liptser-stat1}. Hence, the maximum likelihood
estimator of parameter $\theta$ constructed by observations of $X$ and
$Y$ on the interval $[0,T]$ has the form
\begin{equation}
\label{eq:mle} \hat{\theta}_{T}=\frac{\int_0^Tf(t,X_t,Y_t)\,dX_t}%
{\int_0^Tg^2(t,X_t,Y_t)\,dt} =\theta+
\frac{\int_0^Tg(t,X_t,Y_t)\,dW_t}%
{\int_0^Tg^2(t,X_t,Y_t)\,dt}.
\end{equation}

\begin{thm}\label{th:gen}
Under assumptions~\eqref{eq:ne0}--\eqref{eq:cond}, the estimator $\hat
{\theta}_{T}$ is strongly consistent as
$T\to\infty$.
\end{thm}

\begin{proof}
Note that, under condition~\eqref{eq:sq-int} the process
$M_t=\int_0^tg(s,X_s,Y_s)\,dW_s$
is a square-integrable local martingale with quadratic
variation
$\langle
M\rangle_t=\break\int_0^tg^2(s,X_s,Y_s)\,ds$.
According to the strong law of large numbers for martingales
\cite[Ch.~2, \S~6, Thm.~10, Cor.~1]{liptser-shiryaev1}, under the condition
$\langle M\rangle_{T}\rightarrow\infty$ a.s.\ as $T\to\infty$,
we have that $\frac{M_{T}}{\langle M\rangle_{T}}\rightarrow 0$
a.s.\ as $T\to\infty$. Therefore, it follows from
representation~\eqref{eq:mle} that $\hat{\theta}_T$ is strongly
consistent.
\end{proof}
%
%\begin{remark}
%For example, if the function $\sigma_1(t,x)\sigma_2(t,y)$ is separated
%from zero...
%\end{remark}

\subsection{Linear equation with stochastic volatility}
As an example, let us consider the model
\begin{equation}
\label{eq:ex1-sde} dX_t =\theta X_t\,dt + X_t
\sigma_2(Y_t)\,dW_t, \quad
X_0=x_0\in\R,
\end{equation}
where $W_t$ is a Wiener process, and $Y_t$ is a continuous stochastic
process with values from an open interval $J=(l,r)$ (further, in
examples, we will consider $J=\R$ or $J=(0,+\infty)$).
By Theorem~\ref{th:Lip}, under the assumption
\begin{enumerate}[({A}1)]
\item\label{A0} $\sigma_2(y)$ is locally bounded on $J$,
\end{enumerate}
there exists a unique strong solution of \eqref{eq:ex1-sde}.\vadjust{\eject}

Let $Y$ be a $J$-valued solution of the equation
\begin{equation}
\label{eq:ex1-Y} dY_t =\alpha(Y_t)\,dt +
\beta(Y_t)\,dW_t^1, \quad Y_0=y_0
\in J,
\end{equation}
where $W^1$ is a Wiener process, possibly correlated with $W$.

By $\lloc(J)$ we denote the set of Borel functions $J\to[-\infty,\infty
]$ that are locally integrable on $J$, that is, integrable on compact
subsets of $J$.
By $\lloc(l+)$ we denote the set of Borel functions
$f\colon J\to[-\infty,\infty]$ such that
$\int_l^z\abs{f(y)}\,dy<\infty$
for some $z\in J$.
The notation $\lloc(r-)$ is introduced similarly.

Assume that coefficients $\alpha$ and $\beta$ satisfy the
Engelbert--Schmidt conditions
\begin{enumerate}[({A}1)]
\setcounter{enumi}{1}
\item\label{A1} $\beta(y)\ne0$ for all $y\in J$, and
\item\label{A2} $\beta^{-2},\alpha\beta^{-2}\in\lloc(J)$.
\end{enumerate}

Let us introduce the following notation:
\begin{align*}
\rho(y)&=\exp\set{-2\int_c^y\frac{\alpha(u)}{\beta^2(u)}
\,du}, \quad y\in J,
\\
s(y)&=\int_c^y\rho(u)\,du, \quad y\in\bar
J=[l,r],
\end{align*}
for some $c\in J$.
Assume additionally that
\begin{enumerate}[({A}1)]
\setcounter{enumi}{3}
\item\label{A3} $s(r)=\infty$ or $\frac{s(r)-s}{\rho\beta^2}\notin\lloc(r-)$,
\item\label{A4} $s(l)=-\infty$ or $\frac{s-s(l)}{\rho\beta^2}\notin\lloc(l+)$.
\end{enumerate}

Under (A\ref{A1})--(A\ref{A2}), the SDE \eqref{eq:ex1-Y} has a weak
solution, unique in law, which possibly exits $J$ at some time $\zeta$.
Moreover, $\zeta=\infty$ a.s.\ if and only if conditions (A\ref
{A3})--(A\ref{A4}) are satisfied,\; see, for example, \cite[Prop.~2.6]{miur}.

Assume also that
\begin{enumerate}[({A}1)]
\setcounter{enumi}{5}
\item\label{A5} $\beta^{-2}\sigma_2^{-2}\in\lloc(J)$,
\item\label{A6} one of the following four conditions holds:
\begin{enumerate}[(i)]
\item
$s(r)=\infty$, $s(l)=-\infty$,
\item
$s(r)<\infty$, $s(l)=-\infty$, $\frac{s(r)-s}{\rho\beta^2\sigma
_2^2}\notin\lloc(r-)$,
\item
$s(r)=\infty$, $s(l)>-\infty$,
$\frac{s-s(l)}{\rho\beta^2\sigma_2^2}\notin\lloc(l+)$,
\item
$s(r)<\infty$, $s(l)>-\infty$, $\frac{s(r)-s}{\rho\beta^2\sigma
_2^2}\notin\lloc(r-)$,
$\frac{s-s(l)}{\rho\beta^2\sigma_2^2}\notin\lloc(l+)$,
\end{enumerate}
\item\label{A7} $X_t\sigma(Y_t)\ne0$ a.s., $t\ge0$.
\end{enumerate}

The maximum likelihood estimator~\eqref{eq:mle} for model \eqref
{eq:ex1-sde} equals
\begin{equation}
\label{eq:ex1-est} \hat{\theta}_T=\frac{\int_0^TX_t^{-1}\sigma_2^{-2}(Y_t)\,dX_t}%
{\int_0^T\sigma_2^{-2}(Y_t)\,dt}.
\end{equation}

\begin{thm}\label{th:ex1}
Under assumptions \textup{(A\ref{A0})--(A\ref{A7})}, the
estimator $\hat{\theta}_T$ is strongly consistent as $T\to\infty$.
\end{thm}
\begin{proof}
We need to verify conditions~\eqref{eq:ne0}--\eqref{eq:cond} of
Theorem~\ref{th:gen}.
For model~\eqref{eq:ex1-sde}, they read as follows:
\begin{align}
X_t\sigma_2(Y_t)&\ne0, \quad t\ge0,\quad
\text{a.s.},\label{eq:ex1-ne0}
\\
\int_0^t\sigma_2^{-2}(Y_s)
\,ds&<\infty, \quad t>0,\quad\text{a.s.}, \label{eq:ex1-sq-int}
\\
\int_0^\infty\sigma_2^{-2}(Y_s)
\,ds&=\infty \quad\text{a.s.}\label{eq:ex1-cond}
\end{align}
Note that~\eqref{eq:ex1-ne0} is assumption~(A\ref{A7}).
By~\cite[Thm.~2.7]{miur} the local integrability condition~(A\ref{A5}),
together with (A\ref{A1})--(A\ref{A4}), implies \eqref{eq:ex1-sq-int}.
Further, if assumption~(A\ref{A6})(i) holds, then \eqref{eq:ex1-cond}
is satisfied by \cite[Thm. 2.11]{miur}. In the remaining case
$s(r)<\infty$ or $s(l)>-\infty$, we have that
\[
\varOmega=\set{\lim_{t\uparrow\infty} Y_t=r}\cup\set{\lim
_{t\uparrow\infty
} Y_t=l};
\]
see \cite{miur}.
Moreover, if $s(r)=\infty$, then
$\P (\lim_{t\uparrow\infty} Y_t=r )=0$ by \cite[Prop. 2.4]{miur}.
If $s(r)<\infty$ and
$\frac{s(r)-s}{\rho\beta^2\sigma_2^2}\notin\lloc(r-)$,
then
$\int_0^\infty\sigma_2^{-2}(Y_s)\,ds=\infty$
a.s.\ on $\{\lim_{t\uparrow\infty} Y_t=r\}$
by \cite[Thm. 2.12]{miur}.
The similar statements hold for $\{\lim_{t\uparrow\infty}
Y_t=l\}$. This implies that \eqref{eq:ex1-cond} is satisfied under each
of conditions (ii)--(iv) of assumption (A\ref{A6}).
\end{proof}

Now we consider several examples of the process $Y$, namely the
Bachelier model, the Ornstein--Uhlenbeck model, the geometric Brownian
motion, and the\break Cox--Ingersoll--Ross model.
We concentrate on verification of assumption (A\ref{A6}) for these
models, assuming that other conditions of Theorem~\ref{th:ex1} are satisfied.
%The assumptions (A\ref{A1})--(A\ref{A2}) for them obviously are
%satisfied. The assumptions (A\ref{A3})--(A\ref{A4}) also are known,
%however, they can be checked similarly to (A\ref{A6}) by putting $
%\sigma_2=1$.
%In what follows we assume that assumptions (A\ref{A0}) and (A\ref{A7})
%are satisfied.

\begin{example}[Bachelier model]
Let $Y$ be a solution of the SDE
\[
dY_t =\alpha\,dt + \beta\,dW_t^1, \quad
Y_0=y_0\in\R,
\]
where $\alpha\in\R$ and $\beta\ne0$ are some constants.
Assume that $\sigma_2^{-2}(y)\in\lloc(\R)$ and one of the following
assumptions holds:
\begin{enumerate}[\rm(i)]
\item$\alpha=0$,
\item$\alpha>0$ and $\sigma_2^{-2}(y)\notin\lloc(+\infty)$,
\item$\alpha<0$ and $\sigma_2^{-2}(y)\notin\lloc(-\infty)$.
\end{enumerate}
Then estimator~\eqref{eq:ex1-est} is strongly consistent.
\end{example}
Indeed, in this case, $J=\R$,
\[
\rho(y)=\exp\set{-\frac{2\alpha}{\beta^2}\,y}, \quad\text{and}\quad s(y)=\int
_0^y\exp\set{-\frac{2\alpha}{\beta^2}\,u}\,du.
\]
If $\alpha=0$, then $s(y)=y$,
$s(+\infty)=\infty$, $s(-\infty)=-\infty$, and assumption (A\ref
{A6})(i) is satisfied.
Otherwise, we have\vadjust{\eject}
\[
s(y)=\frac{\beta^2}{2\alpha} \biggl(1-\exp\set{-\frac{2\alpha}{\beta^2}\, y} \biggr).
\]
If $\alpha>0$, then $s(+\infty)=\frac{\beta^2}{2\alpha}$, $s(-\infty
)=-\infty$, and
\[
\frac{s(+\infty)-s(y)}{\rho(y)\beta^2\sigma_2^2(y)} =\frac{1}{2\alpha\sigma_2^2(y)}\notin\lloc(+\infty),
\]
and hence (A\ref{A6})(ii) holds.
The case $\alpha<0$ is considered similarly.

\begin{example}[Ornstein--Uhlenbeck or Vasicek model]
Let $Y$ be a solution of the SDE
\[
dY_t =a(b-Y_t)\,dt + \gamma\,dW_t^1,
\quad Y_0=y_0\in\R,
\]
where $a,b\in\R$, and $\gamma>0$ are some constants.
Assume that $\sigma_2^{-2}\in\lloc(\R)$ and one of the following
assumptions holds:
\begin{enumerate}[\rm(i)]
\item$a\ge0$,
\item$a<0$, $y^{-1}\sigma_2^{-2}(y)\notin\lloc(+\infty)\cup\lloc
(-\infty)$.
\end{enumerate}
Then estimator~\eqref{eq:ex1-est} is strongly consistent.
\end{example}
In this case, we also take $J=\R$. Then
\begin{align*}
\rho(y)&=\exp\set{-2\int_b^y\frac{a(b-u)}{\gamma^2}
\,du} =\exp\set{\frac{a}{\gamma^2}(y-b)^2},
\\
s(y)&=\int_b^y\exp\set{\frac{a}{\gamma^2}(u-b)^2}
\,du.
\end{align*}

If $a\ge0$, then
$\exp\{\frac{a}{\gamma^2}(u-b)^2\}\ge1$, and we get that
$s(+\infty)=\infty$, $s(-\infty)=-\infty$.

If $a<0$, then
\[
s(y)=\frac{\gamma}{\sqrt{-a}}\int_0^{\frac{\sqrt{-a}}{\gamma}(y-b)}
e^{-z^2}dz.
\]
Therefore,
$s(+\infty)=-s(-\infty)=\frac{\gamma\sqrt{\pi}}{2\sqrt{-a}}<\infty$,
and we need to verify (A\ref{A6})(iv).
Since $\int_x^\infty e^{-z^2}dz\sim\frac{1}{2x}e^{-x^2}$ as $x\to\infty
$, we see that
\[
\frac{s(+\infty)-s(y)}{\rho(y)\gamma^2\sigma_2^2(y)} =\frac{\frac{\gamma}{\sqrt{-a}}\int\limits_{\frac{\sqrt{-a}}{\gamma
}(y-b)}^\infty e^{-z^2}dz}{\exp\set{\frac{a}{\gamma^2}(y-b)^2}\gamma
^2\sigma_2^2(y)} \sim\frac{1}{-2a(y-b)\sigma_2^2(y)}
\]
as $y\to\infty$. Then $\frac{s(+\infty)-s(y)}{\rho(y)\gamma^2\sigma
_2^2(y)}\notin\lloc(+\infty)$
if $y^{-1}\sigma_2^{-2}(y)\notin\lloc(+\infty)$.
The condition
$\frac{s-s(-\infty)}{\rho\beta^2\sigma_2^2}\notin\lloc(-\infty)$
is considered similarly.\vadjust{\eject}

\begin{example}[Geometric Brownian motion]
Let $Y$ be a solution of the SDE
\[
dY_t =\alpha Y_t\,dt + \beta Y_t
\,dW_t^1, \quad Y_0=y_0>0,
\]
where $\alpha\,{\in}\,\R$ and $\beta\ne0$ are some constants.
Assume that $y^{-2}\sigma_2^{-2}(y)\,{\in}\,\lloc((0,+\infty))$ and one of
the following assumptions holds:
\begin{enumerate}[\rm(i)]
\item$\beta^2=2\alpha^2$,
\item$\beta^2<2\alpha^2$ and $y^{-1}\sigma_2^{-2}(y)\notin\lloc(+\infty)$,
\item$\beta^2>2\alpha^2$ and $y^{-1}\sigma_2^{-2}(y)\notin\lloc(0+)$.
\end{enumerate}
Then estimator~\eqref{eq:ex1-est} is strongly consistent.
\end{example}
In this case, the process $Y$ is positive, and hence $J=(0,\infty)$.
We have
\begin{align*}
\rho(y)&=\exp\set{-2\int_1^y\frac{\alpha^2}{\beta^2u}
\,du} =y^{-\frac{2\alpha^2}{\beta^2}},
\\
s(y)&=\int_1^yu^{-\frac{2\alpha^2}{\beta^2}}\,du=
\begin{cases}
\frac{y^{1-\frac{2\alpha^2}{\beta^2}}-1}{1-\frac{2\alpha^2}{\beta^2}},
& \beta^2\ne2\alpha^2,\\
\ln y,
& \beta^2=2\alpha^2.
\end{cases} %
\end{align*}
If $\beta^2=2\alpha^2$, then
$s(0)=-\infty$ and $s(+\infty)=\infty$.
If $\beta^2<2\alpha^2$, then
$s(0)=-\infty$,
$s(+\infty)<\infty$, and
\[
\frac{s(+\infty)-s(y)}{\rho(y)\beta^2y^2\sigma_2^2(y)} =\frac{1}{(2\alpha^2-\beta^2)y\sigma_2^2(y)} \notin\lloc(+\infty).
\]
If $\beta^2>2\alpha^2$, then
$s(0)>-\infty$,
$s(+\infty)=\infty$, and
\[
\frac{s(y)-s(0)}{\rho(y)\beta^2y^2\sigma_2^2(y)} =\frac{1}{(\beta^2-2\alpha^2)y\sigma_2^2(y)} \notin\lloc(0+).
\]

\begin{example}[Cox--Ingersoll--Ross model]
Let $Y$ be a solution of the SDE
\[
dY_t =a(b-Y_t)\,dt + \gamma\sqrt{Y_t}
\,dW_t^1, \quad Y_0=y_0\in\R,
\]
where $a$, $b$, $\gamma$ are positive constants, and $2ab\ge
\gamma^2$.
Assume that
\[
y^{-1}\sigma_2^{-2}(y)\in\lloc\bigl((0,+\infty)
\bigr).
\]
Then estimator~\eqref{eq:ex1-est} is strongly consistent.
\end{example}
Under the condition $2ab\ge\gamma^2$, the process $Y$ is positive, and
hence $J=(0,\infty)$. Further,
\begin{align*}
\rho(y)&=\exp\set{-2\int_1^y\frac{a(b-u)}{\gamma^2u}
\,du} %=\exp\set{-\frac{2a}{\gamma^2}(b\ln y-y+1)}
=y^{-\frac{2ab}{\gamma^2}}e^{\frac{2a}{\gamma^2}(y-1)},
\\
s(y)&=e^{-\frac{2a}{\gamma^2}}\int_1^yu^{-\frac{2ab}{\gamma^2}}e^{\frac
{2a}{\gamma^2}u}
\,du.
\end{align*}
Since $u^{-\frac{2ab}{\gamma^2}}e^{\frac{2a}{\gamma^2}u}\to\infty$ as
$u\to\infty$, we see that $s(+\infty)=\infty$.
Moreover, using the inequality $e^{\frac{2a}{\gamma^2}u}\ge1$, we get
\[
s(0)=-e^{-\frac{2a}{\gamma^2}}\int_0^1u^{-\frac{2ab}{\gamma^2}}e^{\frac
{2a}{\gamma^2}u}
\,du \le-e^{-\frac{2a}{\gamma^2}}\int_0^1u^{-\frac{2ab}{\gamma^2}}
\,du =-\infty
\]
since $\frac{2ab}{\gamma^2}>1$.
Thus, assumption~(A\ref{A6})(i) is satisfied.

\section{Simulations}
\label{sec:4}
We illustrate the quality of the estimator $\hat\theta_T$ in model
\eqref{eq:ex1-sde}--\eqref{eq:ex1-Y} by simulation experiments.
We simulate the trajectories of the Wiener processes $W$ and $W^1$ at
the points
$t = 0, h, 2h, 3h,\dots$
and compute the approximate values of the process $Y$ and $X$ as
solutions to SDEs using Euler's approximations.
For each set of parameters, we simulate 100 sample paths with step $h = 0.0001$.
The initial values of the processes are $x_0=y_0=1$, and the true value
of the parameter is $\theta=2$.
The results are reported in Table~\ref{tab:1}.

\begin{table}%[htb]
\tabcolsep=7.3pt
\caption{The means and standard deviations of $\hat\theta_T$}
%\centering
%\footnotesize
%
\begin{tabular}{*{3}{Q}l*{4}{Q}}\hline
&&&& T &&&\\
\cmidrule(l){5-8}
%\cline{5-8}
\alpha(y) & \beta(y) & \sigma_2(y)& & 10 & 50 & 100 & 200\\\hline   \addlinespace
1&1&\abs{y}^{1/4}&Mean&1.9455&1.9431&1.9711&1.9762\\
&&&Std.dev.&0.4260&0.2576&0.2367&0.2022\\
\addlinespace
y&2y&\sqrt{y}&Mean&2.0104&2.0000&2.0000&2.0000\\
&&&Std.dev.&0.1225&5.7\cdot10^{-5}&4.7\cdot10^{-8}&1.6\cdot10^{-14}\\
\addlinespace
y&y&(1+y)^{-1}&Mean&2.0008&2.0001&2.0000&2.0000\\
&&&Std.dev.&0.0769&0.0010&2.2\cdot10^{-12}&1.4\cdot10^{-14}\\
\addlinespace
y&1&2+\sin y&Mean&1.9358&1.9819&1.9927&1.9939\\
&&&Std.dev.&0.5436&0.2437&0.1679&0.1077\\
\addlinespace
-y&1&2+\sin y&Mean&1.9061&1.9684&1.9700&1.9786\\
&&&Std.dev.&0.5994&0.2472&0.1781&0.1254\\
\addlinespace
2-y&\sqrt{y}&\sqrt{y}&Mean&1.9923&2.0039&1.9796&1.9872\\
&&&Std.dev.&0.3540&0.1604&0.1173&0.0782\\
\addlinespace
2-y&\sqrt{y}&y&Mean&2.0830&1.9835&1.9803&1.9886\\
&&&Std.dev.&0.4347&0.1974&0.1205&0.0840\\
\hline
\end{tabular}
\label{tab:1}
\end{table}

%\appendix

\section{Appendix}\label{app}
The next two propositions are taken from \cite{Skorokhod65}.
\begin{prop} \label{prop}
Assume that we have $r$ sequences of stochastic processes
$\xi_n^{(1)}, \ldots, \xi_n^{(r)}$ such that, for all $i=1,\ldots,r$,
\begin{enumerate}[\rm(i)]
\item for every $\delta> 0, \lim_{h \rightarrow
0}\varlimsup_{n\rightarrow\infty}\sup_{\abs{t_1-t_2} \leq h} \P
(|\xi_n^{(i)}(t_1) -\xi_n^{(i)}(t_2)| > \delta )= 0$,
\item$\lim_{C\rightarrow\infty}\lim_{n\rightarrow\infty
}\sup_{0\leq t \leq T}\P (|\xi_n^{(i)}(t)| >C )=0$.
\end{enumerate}
Then, for some sequence $n_k$ we can construct processes
$X_{n_k}^{(1)},\ldots,X_{n_k}^{(r)}$ on the probability space
$(\varOmega', \F', P')$, where $\varOmega'=[0,1]$,
$\F'=\mathcal B ([0,1])$, and $P'$ is the Lebesgue measure, such
that the finite-dimensional distributions of
$X_{n_k}^{(1)},\ldots,X_{n_k}^{(r)}$ coincide with those of
$\xi_{n_k}^{(1)},\ldots,\xi_{n_k}^{(r)}$ and each of the sequences
$X_{n_k}^{(1)},\ldots,X_{n_k}^{(r)}$ converges in probability to
some limit.
\end{prop}
\begin{prop}
Let $\eta_n(t)$ be a sequence of martingales such that $\eta
_n(t)\rightarrow W(t)$ in probability for all $t$ and
$\E\eta_n(t)^2\rightarrow t$ as $n\rightarrow\infty$. Let $f_n(t)$ be
a sequence such that $\int_0^T f_n(t)\,d\eta_n(t)$ exists for all $n$,
$f_n(t)\rightarrow f(t)$ in probability for all $t$, and $\int_0^T
f(t)\,dW(t)$ exists.
Suppose that, in addition, the following conditions hold:
\begin{enumerate}[\rm(i)]
\item
for all $\varepsilon>0$, there exists $C >0$ such that, for all $n$,
\[
\P \Bigl(\sup_{0 \leq t \leq T}\big|f_n(t)\big| >C \Bigr)\leq
\varepsilon,
\]
\item
for all $\varepsilon>0$,
\[
\lim_{h\rightarrow0}\lim_{n\rightarrow\infty}\sup
_{\abs{t_1-t_2} \leq
h}\P\bigl(\big|f_n(t_2)-f_n(t_1)\big|
> \varepsilon\bigr)=0.
\]
\end{enumerate}
Then $\int_0^T f_n(t)\,d \eta_n(t)\rightarrow\int_0^T f(t)\,dW(t)$ in
probability.
\end{prop}
In the next two lemmas, we modify the corresponding auxiliary results
from \cite{Skorokhod65} to Eq.~\eqref{eq:intro} with multiplicative
diffusion. Consider a sequence of partitions $0=t_0^n <t_1^n
<\cdots<t_n^n=T$ of $[0,T]$ such that $\lim_{n\rightarrow\infty} \max_k (t_{k+1}^n-t_k^n )=0$.
Define $\xi_k^n$ by $\xi_0^n=X(0)$ and
\[
\xi_{k+1}^n=\xi_k^n+a
\bigl(t_k^n,\xi_k^n\bigr)\Delta
t_k^n+\sigma_1\bigl(t_k^n,
\xi_k^n\bigr) \sigma_2\bigl(t_k^n,Y
\bigl(t_k^n\bigr)\bigr)\Delta W_k^n.
\]
\begin{lemma}\label{lem1}
The random variables $\sup_k|\xi_k^n|$ are bounded in
probability uniformly w.r.t.\ $n$.
\end{lemma}
\begin{proof}
Let $\eta_0^n=\xi_0^n\; \ind_{\abs{\xi_0^n} \leq N}$ and
\begin{equation}
\label{eq:2.} \eta_{k+1}^n=\eta_k^n+a^N
\bigl(t_k^n,\eta_k^n\bigr)\Delta
t_k^n+\sigma^N_1
\bigl(t_k^n,\eta_k^n\bigr)
\sigma_2\bigl(t_k^n,Y\bigl(t_k^n
\bigr)\bigr)\Delta W_k^n,
\end{equation}
where $a^N(t,x)=a(t,x)\ind_{\abs{x} \leq N}$,
$\sigma^N_1(t,x)=\sigma_1(t,x)\ind_{\abs{x} \leq N}$.
If
$|\eta_k^n| > N$, then $\eta_{k+1}^n=\eta_k^n$, and if
$|\eta_k^n|\leq N$, then
\[
\abs{\eta_{k+1}^n} \leq N + \big|a^N
\bigl(t_k^n,\eta_k^n\bigr)\big|
\Delta t_k^n+\big|\sigma^N_1
\bigl(t_k^n,\eta_k^n\bigr)
\sigma_2\bigl(t_k^n,Y\bigl(t_k^n
\bigr)\bigr)\big|\abs{\Delta W_k^n}.
\]
Then, for
any $ 1 \leq k \leq n$,
\[
\abs{\eta_k^n}\leq N + \sum
_{r=0}^{k-1}\big|a^N\bigl(t_r^n,
\eta_r^n\bigr)\big| \Delta t_r^n+
\sum_{r=0}^{k-1} \big|\sigma^N_1
\bigl(t_r^n,\eta_r^n\bigr)
\sigma_2\bigl(t_r^n,Y\bigl(t_r^n
\bigr)\bigr)\big|\abs{\Delta W_r^n}
\]
and is square-integrable.
Furthermore,
\begin{align*}
\E\abs{\eta_{k+1}^n}^2&=\E\abs{
\eta_k^n}^2+2 \E \bigl(a^N
\bigl(t_k^n,\eta _k^n\bigr)
\eta_k^n \Delta t_k^n \bigr) +\E
a^N\bigl(t_k^n,\eta_k^n
\bigr)^2 \bigl(\Delta t_k^n
\bigr)^2
\\
&\quad+\E \bigl( \bigl(\sigma^N_1\bigl(t_k^n,
\eta_k^n\bigr) \bigr)^2 \bigl(
\sigma_2\bigl(t_k^n,Y\bigl(t_k^n
\bigr)\bigr)\bigr)^2 \bigl(\Delta W_k^n
\bigr)^2 \bigr)
\\
&\leq\E\abs{\eta_k^n}^2+2 \E
\bigl(a^N\bigl(t_k^n,\eta_k^n
\bigr)\eta_k^n \Delta t_k^n
\bigr)+\E a^N\bigl(t_k^n,
\eta_k^n\bigr)^2 \bigl(\Delta
t_k^n\bigr)^2
\\
&\quad+C^2 \E \bigl( \bigl(\sigma^N_1
\bigl(t_k^n,\eta_k^n \bigr)
\bigr)^2 \Delta t_k^n \bigr).
\end{align*}
Then there exists a constant $H=H(T,K)$ such that
\begin{align*}
\E\abs{\eta_{k+1}^n}^2&\leq\E\abs{
\eta_k^n}^2\bigl(1+H\Delta
t_k^n\bigr)+H\Delta t_k^n \leq
\E\abs{\eta_k^n}^2 e^{H\Delta t_k^n}+H\Delta
t_k^n,
\\
\E\abs{\eta_{k+1}^n}^2+1&\leq \bigl(\E\abs{
\eta_k^n}^2+1 \bigr) e^{H\Delta t_k^n}\leq
\bigl(\E\abs{\eta_0^n}^2+1 \bigr)
e^{HT},
\\
\E\abs{\eta_{k+1}^n}^2&\leq \bigl(\E\abs{
\eta_0^n}^2+1 \bigr) e^{HT}-1.
\end{align*}
We have that
\begin{align*}
\sup_k\abs{\eta_k^n} \leq\abs{
\eta_0^n } + \sum_{j=0}^{n-1}
\big|a^N \bigl(t_j^n,\eta_j^n
\bigr)\big| \Delta t_j^n
+\sup_k\abs{\sum_{j=0}^{k-1}
\sigma^N_1 \bigl(t_j^n,
\eta_j^n \bigr) \sigma_2
\bigl(t_j^n,Y \bigl(t_j^n \bigr)
\bigr) \Delta W_j^n}.
\end{align*}
So, since
\begin{align*}
&\E\sup_{0 \leq r \leq n}\abs{\sum_{k=0}^{r}
\sigma^N_1\bigl(t_k^n,
\eta_k^n\bigr) \sigma_2\bigl(t_k^n,Y
\bigl(t_k^n\bigr)\bigr) \Delta W_k^n}^2
\\
&\quad\leq4 \sum_{k=0}^{n} \E\big|
\sigma^N_1\bigl(t_k^n,
\eta_k^n\bigr) \sigma_2\bigl(t_k^n,Y
\bigl(t_k^n\bigr)\bigr)\big|^2 \Delta
t_k^n \leq4 K \sum_{k=0}^{n}
\bigl(\E\abs{\eta_k^n}^2+1\bigr)\Delta
t_k^n
\end{align*}
and
\[
\E \Biggl( \sum_{k=0}^{n-1}
\big|a^N\bigl(t_k^n,\eta_k^n
\bigr)\big| \Delta t_k^n \Biggr)^2\leq KT\sum
_{k=0}^{n-1}\bigl(\E\abs{\eta_k^n}^2+1
\bigr)\Delta t_k^n,
\]
we get
\[
\E\sup_k \big|\eta_k^n
\big|^2 \leq A+B\big| \eta_0^n\big|^2.
\]
We have that $ \eta_0^n$ is bounded uniformly w.r.t.\ $n$ and $N$. Then
$\sup_k|\eta_k^n|^2$ is bounded in $L_2$ uniformly w.r.t.\ $n$, $N$.

For $\sup_k |\eta_k^n|< N$, we have $\sup_k|\eta_k^n|=\sup_k
|\xi_k^n|$.
Hence, $\sup_k|\xi_k^n|$ is bounded in probability uniformly
w.r.t.\ $n$.
\end{proof}
\begin{remark}\label{r1}
Using Lemma \ref{lem1}, we have that, for all $\varepsilon>0$, there
exists $N>0$ such that,
for every $n \ge N, \; \P(\sup_k |\xi_k^n -\eta_k^n|
>0)<\varepsilon$.
\end{remark}
\begin{lemma}\label{lem 2}
Let $\xi_n(t)=\xi_k^n$ for $t\in[t_k^n,t_{k+1}^n)$. Then, for all
$\delta>0$,
\[
\lim_{h\rightarrow0}\varlimsup_{n\rightarrow\infty}\sup_{\abs
{t_1-t_2}\leq h}
\P\bigl(\big|\xi_n(t_1) -\xi_n(t_2)\big|
>\delta\bigr)=0.
\]
\end{lemma}
\begin{proof}
Let $\eta_n(t)=\eta_k^n$, $t\in[t_k^n,t_{k+1}^n)$. Then
\begin{align*}
\sup_{\abs{t_1-t_2}\leq h} \P\bigl(\big|\xi_n(t_1) -
\xi_n(t_2) \big| >\delta \bigr)&\leq\sup_{\abs{ t_1-t_2 } \leq h}
\P\bigl(\big| \eta_n(t_1) -\eta_n(t_2)
\big| >\delta\bigr)
\\
&\quad+ \P\Bigl(\sup_k \big| \xi_k^n -
\eta_k^n \big| >0\Bigr).
\end{align*}
From (\ref{eq:2.}) and the boundedness of $a^N$, $\sigma_1^N$, and
$\sigma_2$ we have that
\[
\lim_{h\rightarrow0}\varlimsup_{n\rightarrow\infty}\sup_{\abs
{t_1-t_2} \leq h}
\P\bigl(\big|\eta_n(t_1) -\eta_n(t_2)\big|
>\delta\bigr)=0.
\]
Therefore,
\[
\lim_{h\rightarrow0}\varlimsup_{n\rightarrow\infty}\sup_{\abs
{t_1-t_2} \leq h}
\P\bigl(\big| \xi_n(t_1) -\xi_n(t_2)
\big| >\delta\bigr) \leq \varlimsup_{n\rightarrow\infty}\P\Bigl(\sup_k
\big| \xi_k^n -\eta_k^n \big| >0
\Bigr).
\]
The proof follows now from Remark \ref{r1}.
\end{proof}

%% Acknowledgements %%
%%%%%%%%%%%%%%%%%%%%%%
\section*{Acknowledgment}
The authors are grateful to the anonymous referee for his useful
remarks and suggestions, which contributed to a substantial improvement
of the text.

\end{document}